\documentclass[12pt]{amsproc}%
\usepackage{hyperref}
\usepackage{geometry}
\usepackage{graphicx}
\usepackage{amsthm}
\usepackage{amssymb}
\usepackage{amsmath}
\usepackage{amsfonts}%
\providecommand{\U}[1]{\protect\rule{.1in}{.1in}}
\theoremstyle{plain}
\newtheorem{thm}{Theorem}[section]
\newtheorem{cor}[thm]{Corollary}
\newtheorem{lem}[thm]{Lemma}
\newtheorem{prop}[thm]{Proposition}
\newtheorem{defn}[thm]{Definition}
\newtheorem{exa}[thm]{Example}
\newtheorem{rem}[thm]{Remark}

\begin{document}
\title{On Graded $\phi$-$1$-absorbing prime ideals}
\author{Mashhoor \textsc{Refai}}
\address{President of Princess Sumaya University for Technology, Amman, Jordan}
\email{m.refai@psut.edu.jo}
\author{Rashid \textsc{Abu-Dawwas}}
\address{Department of Mathematics, Yarmouk University, Irbid, Jordan}
\email{rrashid@yu.edu.jo}
\author{\"{U}nsal \textsc{Tekir}}
\address{Department of Mathematics, Marmara University, Istanbul, Turkey}
\email{utekir@marmara.edu.tr}
\author{Suat \textsc{Ko\c{c}}}
\address{Department of Mathematics, Marmara University, Istanbul, Turkey}
\email{suat.koc@marmara.edu.tr}
\author{Roa'a \textsc{Awawdeh}}
\address{Department of Mathematics, Yarmouk University, Irbid, Jordan}
\email{2018105027@ses.yu.edu.jo}
\author{Eda \textsc{Yildiz}}
\address{Department of Mathematics, Yildiz Technical University, Istanbul, Turkey}
\email{edyildiz@yildiz.edu.tr}

\begin{abstract}
Let $G$ be a group, $R$ be a $G$-graded commutative ring with nonzero unity
and $GI(R)$ be the set of all graded ideals of $R$. Suppose that
$\phi:GI(R)\rightarrow GI(R)\cup\{\emptyset\}$ is a function. In this article,
we introduce and study the concept of graded $\phi$-$1$-absorbing prime
ideals. A proper graded ideal $I$ of $R$ is called a graded $\phi$%
-$1$-absorbing prime ideal of $R$ if whenever $a,b,c$ are homogeneous nonunit
elements of $R$ such that $abc\in I-\phi(I)$, then $ab\in I$ or $c\in I$.
Several properties of graded $\phi$-$1$-absorbing prime ideals have been examined.

\end{abstract}
\maketitle

\section{Introduction}

Throughout this article, $G$ will be a group with identity $e$ and $R$ be a
commutative ring having a nonzero unity $1$. Then $R$ is called a
$G$\textit{-graded ring} if $R=\bigoplus_{g\in G}R_{g}$ with $R_{g}%
R_{h}\subseteq R_{gh}$ for all $g,h\in G$ where $R_{g}$ is an additive
subgroup of $R$ for all $g\in G$. The elements of $R_{g}$ are called
\textit{homogeneous of degree }$g$. If $a\in R$, then $a$ can be written
uniquely as a finite sum $\sum_{g\in G}a_{g}$, where $a_{g}$ is the component
of $a$ in $R_{g}$. Note that $R_{e}$ is a subring of $R$ and $1\in R_{e}$. The
set of all homogeneous elements of $R$ is denoted by $h(R)=\bigcup_{g\in
G}R_{g}$. Let $P$ be an ideal of a graded ring $R$. Then $P$ is called a
\textit{graded ideal} if $P=\bigoplus_{g\in G}(P\cap R_{g})$, or equivalently,
$a=\sum_{g\in G}a_{g}\in P\ $implies that $a_{g}\in P$ for all $g\in G$. It is
not necessary that every ideal of a graded ring is a graded ideal. For
instance, let $R=k[X]\ $where $k$ is a field. Then $R\ $is a $%
\mathbb{Z}
$-graded ring where $R_{n}=0$ if $n<0,\ R_{0}=k\ $and $R_{n}=kX^{n}\ $if
$n>0.$\ Then $I=(X+1)\ $is not a graded ideal since $1+X\in I$ but $1\notin
I.\ $ We will denote the set of all graded ideals of $R$ by $GI(R)$. For more
details and terminology, see \cite{Hazart, Nastasescue}.

For many years, various classes of graded ideals have been established such as
graded prime, graded primary, graded absorbing ideals, and etc. All of them
play an important performance when characterizing graded rings. The concept of
graded prime ideals and its generalizations have an important place in graded
commutative algebra since they are used in recognizing the structure of graded
rings. Recall that a proper graded ideal $I$ of $R$ is said to be a
\textit{graded prime ideal} if whenever $a,b\in h(R)$ such that $ab\in I$,
then either $a\in I$ or $b\in I$ (\cite{Refai Hailat}). The significance of
graded prime ideals led many researchers to work on graded prime ideals and
its generalizations. See for example, \cite{Dawwas, Bataineh Dawwas, Refai
Zoubi}. In \cite{Atani Prime}, Atani introduced the notion of graded weakly
prime ideal which is a generalization of graded prime ideals. A proper graded
ideal $I$ of $R$ is said to be a \textit{graded weakly prime ideal} of $R$ if
whenever $a,b\in h(R)$ such that $0\neq ab\in I$, then $a\in I$ or $b\in I$.
It is obvious that every graded prime ideal is graded weakly prime but the
converse is not true in general. For instance, consider the $%
\mathbb{Z}
$-graded ring $R=%
\mathbb{Z}
_{4}[X]\ $and the ideal $I=(0).$\ Then $I\ $is clearly a graded weakly prime
ideal. However, $I\ $is not a graded prime ideal since $\overline{2}%
\cdot\overline{2}X=\overline{0}\ $but $\overline{2}\ $and $\overline{2}X\notin
I.\ $ Later, Al-Zoubi, Abu-Dawwas and Ceken in \cite{Zoubi Dawwas Ceken}
introduced the notion of graded $2$-absorbing ideals. A nonzero proper graded
ideal $I$ of $R$ is called a\textit{ graded }$2$\textit{-absorbing ideal} if
$abc\in I$ implies $ab\in I$ or $ac\in I$ or $bc\in I$ for each $a,b,c\in
h(R)$. Note that every graded prime ideal is also a graded $2$-absorbing
ideal. After this, graded $2$-absorbing version of graded ideals and many
generalizations of graded $2$-absorbing ideals attracted considerable
attention by many researchers in \cite{Dawwas Bataineh Shashan, Soheilnia
Darani, RaTeShKo}. In \cite{Jaber Bataineh Khashan}, the authors defined the
notion of graded almost prime ideals. A proper graded ideal $I$ of $R$ is said
to be \textit{graded almost prime} if for $a,b\in h(R)$ such that $ab\in
I-I^{2}$, then either $a\in I$ or $b\in I$. Also, in \cite{Zoubi Dawwas
Ceken}, the authors defined and studied graded \textit{weakly }$2$%
\textit{-absorbing ideals} which is a generalization of graded weakly prime
ideals. A proper graded ideal $I$ of $R$ is called a \textit{graded weakly
}$2$\textit{-absorbing ideal} if $0\neq abc\in I$ implies $ab\in I$ or $ac\in
I$ or $bc\in I$ for each $a,b,c\in h(R)$. In \cite{Alshehry}, Alshehry and
Abu-Dawwas defined a new class of graded prime ideals. A proper graded ideal
$I$ of $R$ is called \textit{graded }$\phi$\textit{-prime ideal} if whenever
$ab\in I-\phi(I)$ for some $a,b\in h(R)$, then either $a\in I$ or $b\in I$,
where $\phi:GI(R)\rightarrow GI(R)\cup\{\emptyset\}$ is a function. They
proved that a graded prime ideal and a graded $\phi$-prime ideal have some
similar properties.

Recently, in \cite{Dawwas Yildiz}, the notion of graded $1$-absorbing prime
ideals has been introduced and studied. This class of graded ideals is a
generalization of graded prime ideals. A proper graded ideal $I$ of $R$ is
called a \textit{graded }$1$\textit{-absorbing prime ideal} if whenever
$abc\in I$ for some nonunits $a,b,c\in h(R)$, then either $ab\in I$ or $c\in
I$. Note that every graded prime ideal is graded $1$-absorbing prime and every
graded $1$-absorbing prime ideal is graded $2$-absorbing ideal. The converses
are not true. More currently, in \cite{Tekir Koc Dawwas Yildiz}, the notion of
graded weakly $1$-absorbing prime ideals which is a generalization of graded
$1$-absorbing prime ideals has been introduced and investigated. A proper
graded ideal $I$ of $R$ is called a \textit{graded weakly }$1$%
\textit{-absorbing prime ideal} if whenever $0\neq abc\in I$ for some nonunits
$a,b,c\in h(R)$, then either $ab\in I$ or $c\in I$.

In this article, we act in accordance with \cite{Yildiz Tekir Koc} to define
and study graded $\phi$-$1$-absorbing prime ideals as a new class of graded
ideals which is a generalization of graded $1$-absorbing prime ideals. A
proper graded ideal $I$ of $R$ is called a \textit{graded }$\phi$\textit{-}%
$1$\textit{-absorbing prime ideal} of $R$ if whenever $a,b,c\in h(R)$ are
nonunits such that $abc\in I-\phi(I)$, then $ab\in I$ or $c\in I$. Among
several results, an example of a graded weakly $1$-absorbing prime ideal that
is not graded $1$-absorbing prime has been given (Example \ref{Example 2}).
Also, an example of a graded weakly $1$-absorbing prime ideal that is not
graded weakly prime has been introduced (Example \ref{Example 4}). In Theorem
\ref{Theorem 1}, we give a characterization on graded $\phi$-$1$-absorbing
prime ideals. We introduce the concept of $g$-$\phi$-$1$-absorbing prime
ideals. A graded ideal $I$ of $R$ with $I_{g}\neq R_{g}$ is said to be a
$g$\textit{-}$\phi$\textit{-}$1$\textit{-absorbing prime ideal of }$R$ if
whenever $a,b,c\in R_{g}$ such that $abc\in I$, then either $ab\in I$ or $c\in
I$. In Theorem \ref{Theorem 1 (1)}, we give a characterization of $g$-$\phi
$-$1$-absorbing prime ideals.  We show
that if $I$ is a graded $\phi$-$1$-absorbing prime ideal of $R$, then
$I/\phi(I)$ is a graded weakly $1$-absorbing prime ideal of $R/\phi(I)$
(Theorem \ref{Theorem 5 (i)}). On the other hand, we prove that if $I/\phi(I)$
is a graded weakly $1$-absorbing prime ideal of $R/\phi(I)$ and $U(R/\phi
(I))=\left\{  a+\phi(I):a\in U(R)\right\}  $, then $I$ is a graded $\phi$%
-$1$-absorbing prime ideal of $R$ (Theorem \ref{Theorem 5 (ii)}). In Theorem
\ref{Theorem 6}, we study graded $\phi$-$1$-absorbing prime ideals over
multiplicative sets. In Theorems \ref{Theorem 7 (1)}, \ref{Theorem 7 (2)} and
\ref{Theorem 7 (3)}, we study graded $\phi$%
-$1$-absorbing prime ideals over cartesian products of graded rings. Finally,
we introduce and study the concept of graded von Neumann regular rings. A
graded ring $R$ is said to be a \textit{graded von Neumann regular ring} if
for each $a\in R_{g}$, there exists $x\in R_{g^{-1}}$ such that $a=a^{2}x$ \cite{Nastasescue}. In
particular, we prove that if $R$ is a graded von Neumann regular ring and
$x\in h(R)$, then $Rx$ is a graded almost $1$-absorbing prime ideal of $R$
(Theorem \ref{5}).

\section{Graded $\phi$-$1$-absorbing prime ideals}

In this section, we introduce and study the concept of graded $\phi$%
-$1$-absorbing prime ideals.

\begin{defn}
Let $R$ be a graded ring and $\phi:GI(R)\rightarrow GI(R)\cup\{\emptyset\}$ be
a function. A proper graded ideal $I$ of $R$ is called a graded $\phi$%
-$1$-absorbing prime ideal of $R$ if whenever $a,b,c\in h(R)$ are nonunits
such that $abc\in I-\phi(I)$, then $ab\in I$ or $c\in I$.
\end{defn}

\begin{rem}
\label{Example 1} The following notations are used for the rest of the
article, they are types of graded $1$-absorbing prime ideals corresponding to
$\phi_{\alpha}$.

\begin{enumerate}
\item $\phi_{\emptyset}(I)=\emptyset$ (graded $1$-absorbing prime ideal)

\item $\phi_{0}(I)=\{0\}$ (graded weakly $1$-absorbing prime ideal)

\item $\phi_{1}(I)=I$ (any graded ideal)

\item $\phi_{2}(I)=I^{2}$ (graded almost $1$-absorbing prime ideal)

\item $\phi_{n}(I)=I^{n}$ (graded $n$-almost $1$-absorbing prime ideal)

\item $\phi_{\omega}(I)=\bigcap_{n=1}^{\infty}I^{n}$ (graded $\omega$%
-$1$-absorbing prime ideal)
\end{enumerate}
\end{rem}

\begin{rem}
(1)\ Since $I-\phi(I)=I-(I\bigcap\phi(I))$ for any graded ideal $I$, without
loss of generality, throughout this article, we suppose that $\phi(I)\subseteq
I$.

(2) For functions $\phi,\psi:GI(R)\rightarrow GI(R)\cup\{\emptyset\}$, we
write $\phi\leq\psi$ if $\phi(I)\subseteq\psi(I)$ for all $I\in GI(R)$.
Obviously, therefore, we have the next order:
\end{rem}

\begin{rem}
$\phi_{\emptyset}\leq\phi_{0}\leq\phi_{\omega}\leq\cdots\leq\phi_{n+1}\leq
\phi_{n}\leq\cdots\leq\phi_{2}\leq\phi_{1}$.
\end{rem}

\begin{prop}
\label{Proposition 1}Let $R$ be a graded ring, $\phi,\psi:GI(R)\rightarrow
GI(R)\cup\{\emptyset\}$ be two functions with $\phi\leq\psi$ and $I$ be a
proper graded ideal of $R$.

\begin{enumerate}
\item If $I$ is a graded $\phi$-$1$-absorbing prime ideal of $R$, then $I$ is
a graded $\psi$-$1$-absorbing prime ideal of $R$.

\item $I$ is a graded $1$-absorbing prime ideal of $R\Rightarrow$ $I$ is a
graded weakly $1$-absorbing prime ideal of $R\Rightarrow I$ is a graded
$\omega$-$1$-absorbing prime ideal of $R\Rightarrow I$ is a graded $n$-almost
$1$-absorbing prime ideal of $R$ for each $n\geq2\Rightarrow I$ is a graded
almost $1$-absorbing prime ideal of $R$.

\item $I$ is a graded $n$-almost $1$-absorbing prime ideal of $R$ for each
$n\geq2$ if and only if $I$ is a graded $\omega$-$1$-absorbing prime ideal of
$R$.

\item Every graded $\phi$-prime ideal of $R$ is a graded $\phi$-$1$-absorbing
prime ideal of $R$.
\end{enumerate}
\end{prop}

\begin{proof}
$(1):\ $It is clear.

$(2):\ $It follows from (1) and $\phi_{\emptyset}\leq\phi_{0}\leq\phi_{\omega
}\leq\cdots\leq\phi_{n+1}\leq\phi_{n}\leq\cdots\leq\phi_{2}\leq\phi_{1}$ in
Remark \ref{Example 1}.

$(3):\ $By (2), if $I$ is a graded $\omega$-$1$-absorbing prime ideal of $R$,
then $I$ is a graded $n$-almost $1$-absorbing prime ideal of $R$ for each
$n\geq2$. Assume that $I$ is a graded $n$-almost $1$-absorbing prime ideal of
$R$ for each $n\geq2$. Let $abc\in I-\bigcap_{n=1}^{\infty}I^{n}$ for some
nonunits $a,b,c\in h(R)$. Then there exists $r\geq2$ such that $abc\notin
I^{r}$. Since $I$ is a graded $r$-almost $1$-absorbing prime ideal of $R$ and
$abc\in I-I^{r}$, then either we have $ab\in I$ or $c\in I$.

$(4):\ $It is obvious.
\end{proof}

The next example introduces a graded weakly $1$-absorbing prime ideal that is
not a graded $1$-absorbing prime.

\begin{exa}
\label{Example 2}Consider $R=\mathbb{Z}_{pq^{2}}[i]$, where $p$, $q$ are two
distinct primes, and $G=\mathbb{Z}_{2}$. Then $R$ is $G$-graded by
$R_{0}=\mathbb{Z}_{pq^{2}}$ and $R_{1}=i\mathbb{Z}_{pq^{2}}$. As $\overline
{q}^{2}\in R_{0}$, $I=\langle\overline{q}^{2}\rangle$ is a graded ideal of
$R$. Since $\overline{p}, \overline{q}\in R_{0}\subseteq h(R)$ are nonunits
with $\overline{pqq}\in I$ while $\overline{pq}\notin I$ and $\overline
{q}\notin I$, $I$ is not a graded $1$-absorbing prime ideal of $R$. On the
other hand, we prove that $I$ is a graded weakly $1$-absorbing prime ideal of
$R$. Let $\overline{0}\neq\overline{abc}\in I$ for some nonunits $\overline
{a}, \overline{b}, \overline{c}\in h(R)$. Then $q^{2}$ divides $abc$ but
$pq^{2}$ does not divide $abc$.

\underline{Case (1):} $\overline{a}, \overline{b}, \overline{c}\in R_{0}$.

Since $\overline{a}, \overline{b}, \overline{c}$ are nonunits, $p$ or $q$ must
divide $a$, $b$ and $c$. If $p$ divides $a, b $ or $c$, then $pq^{2}$ divides
$abc$ which is a contradiction. So, $q^{2}$ divides $ab$ and so $\overline
{ab}\in I$. Therefore, $I$ is a graded weakly $1$-absorbing prime ideal of $R$.

\underline{Case (2):} $\overline{a}, \overline{b}\in R_{0}, \overline{c}\in
R_{1}$.

In this case, $\overline{c}=i\overline{\alpha}$ for some $\overline{\alpha}\in
R_{0}$. As $\overline{c}$ is nonunit, $\overline{\alpha}$ is nonunit with
$abc=iab\alpha$ and $pq^{2}$ does not divide $ab\alpha$. Since $q^{2}$ divides
$abc$, $iab\alpha=q^{2}(x+iy)$ for some $x, y\in R_{0}$, and then
$ab\alpha=q^{2}y$ which implies that $q^{2}$ divides $ab\alpha$. Similarly as
in case (1), we have that $\overline{ab}\in I$. Therefore, $I$ is a graded
weakly $1$-absorbing prime ideal of $R$.

\underline{Case (3):} $\overline{a}\in R_{0}$, $\overline{b}, \overline{c}\in
R_{1}$.

In this case, $\overline{b}=i\overline{\alpha}$ and $\overline{c}%
=i\overline{\beta}$ for some $\overline{\alpha}, \overline{\beta}\in R_{0}$.
As $\overline{b}$ and $\overline{c}$ are nonunits, $\overline{\alpha}$ and
$\overline{\beta}$ are nonunits with $abc=-a\alpha\beta$ and $pq^{2}$ does not
divide $a\alpha\beta$. Since $q^{2}$ divides $abc$, $-a\alpha\beta
=q^{2}(x+iy)$ for some $x, y\in R_{0}$, and then $-a\alpha\beta=q^{2}x$ which
implies that $q^{2}$ divides $a\alpha\beta$. Similarly as in case (1), we have
that $\overline{a\alpha}\in I$ and then $\overline{ab}\in I$. Therefore, $I$
is a graded weakly $1$-absorbing prime ideal of $R$.

\underline{Case (4):} $\overline{a}, \overline{b}, \overline{c}\in R_{1}$.

In this case, $\overline{a}=i\overline{\alpha}$, $\overline{b}=i\overline
{\beta}$ and $\overline{c}=i\overline{\gamma}$ for some $\overline{\alpha
},\overline{\beta},\overline{\gamma}\in R_{0}$. As $\overline{a}$,
$\overline{b}$ and $\overline{c}$ are nonunits, $\overline{\alpha}$,
$\overline{\beta}$ and $\overline{\gamma}$ are nonunits with $abc=-i\alpha
\beta\gamma$ and $pq^{2}$ does not divide $\alpha\beta\gamma$. Since $q^{2}$
divides $abc$, $-i\alpha\beta\gamma=q^{2}(x+iy)$ for some $x,y\in R_{0}$, and
then $-\alpha\beta\gamma=q^{2}y$ which implies that $q^{2}$ divides
$\alpha\beta\gamma$. Similarly as in case (1), we have that $\overline
{\alpha\beta}\in I$ and then $\overline{ab}\in I$. Therefore, $I$ is a graded
weakly $1$-absorbing prime ideal of $R$.

Since the other cases are similar to one of the above cases, $I\ $is a graded
weakly 1-absorbing prime ideal of $R.$
\end{exa}

The next example introduces a graded $\omega$-$1$-absorbing prime ideal that
is not graded weakly $1$-absorbing prime.

\begin{exa}
\label{Example 3}Consider $R=\mathbb{Z}_{2}\times\mathbb{Z}_{2}\times
\mathbb{Z}_{2}\times\mathbb{Z}_{2}$ and the trivial graduation of $R$ by any
group $G$, that is $R_{e}=R$ and $R_{g}=\{0\}$ for $g\in G-\{e\}$. Now,
$I=\mathbb{Z}_{2}\times\{\overline{0}\}\times\{\overline{0}\}\times
\{\overline{0}\}$ is a graded ideal of $R$ satisfies $I^{2}=I$, and then
$I^{n}=I$ for all $n\geq2$, and hence $I$ is a graded $\omega$-$1$-absorbing
prime ideal of $R$. On the other hand, $I$ is not a graded weakly
$1$-absorbing prime ideal of $R$ since $a=(\overline{1},\overline{1}%
,\overline{1},\overline{0}),b=(\overline{1},\overline{1},\overline
{0},\overline{1})$ and $c=(\overline{1},\overline{0},\overline{1},\overline
{1})\in h(R)$ are nonunits with $0\neq abc\in I$ while $ab,c\notin I$.
\end{exa}

The next example introduces a graded weakly $1$-absorbing prime ideal that is
not graded weakly prime.

\begin{exa}
\label{Example 4}Consider $R=\mathbb{Z}_{pq^{2}}[i]$, where $p$, $q$ are two
distinct primes, and $G=\mathbb{Z}_{2}$. Then $R$ is $G$-graded by
$R_{0}=\mathbb{Z}_{pq^{2}}$ and $R_{1}=i\mathbb{Z}_{pq^{2}}$. By Example
\ref{Example 2}, $I=\langle\overline{q}^{2}\rangle$ is a graded weakly
$1$-absorbing prime ideal of $R$. On the other hand, $I$ is not a graded
weakly prime ideal of $R$ since $\overline{q}\in h(R)$ with $\overline{0}%
\neq\overline{qq}\in I$ while $\overline{q}\notin I$.
\end{exa}

A graded ring $R$ is said to be graded local if it has a unique graded maximal
ideal $\mathfrak{m}$, and it is denoted by $(R,\mathfrak{m})$.

\begin{prop}
\label{Remark 1} Let $(R,\mathfrak{m})$ be a graded local ring and $I$ be a
proper graded ideal of $R$. If $\mathfrak{m}^{2}\subseteq I$, then $I$ is a
graded $1$-absorbing prime ideal of $R$.
\end{prop}

\begin{proof}
Let $abc\in I$ for some nonunits $a,b,c\in h(R)$. Then $a,b,c\in\mathfrak{m}$,
which implies that $ab\in\mathfrak{m}^{2}\subseteq I$. Therefore, $I$ is a
graded $1$-absorbing prime ideal of $R$.
\end{proof}

\begin{cor}
\label{Remark 1 (1)} Let $(R,\mathfrak{m})$ be a graded local ring. Then
$\mathfrak{m}^{2}$ is a graded $1$-absorbing prime ideal of $R$.
\end{cor}

\begin{proof}
By [\cite{Dawwas Yildiz}, Lemma 1], $\mathfrak{m}^{2}$ is a proper graded
ideal of $R$, and then $\mathfrak{m}^{2}$ is a graded $1$-absorbing prime
ideal of $R$ by Proposition \ref{Remark 1}.
\end{proof}

\begin{prop}
\label{Lemma 1} Let $(R,\mathfrak{m})$ be a graded local ring and $I$ be a
proper graded ideal of $R$. If $\mathfrak{m}^{3}\subseteq\phi(I)$, then $I$ is
a graded $\phi$-$1$-absorbing prime ideal of $R$.
\end{prop}

\begin{proof}
Suppose that $I$ is not a graded $\phi$-$1$-absorbing prime ideal of $R$. Then
there exist nonunit elements $a,b,c\in h(R)$ such that $abc\in I-\phi(I)$ but
$ab\notin I$ and $c\notin I$. Since $a,b,c$ are nonunits, they are elements of
$\mathfrak{m}$, and then $abc\in\mathfrak{m}^{3}\subseteq\phi(I)$, which is a
contradiction. Hence, $I$ is a graded $\phi$-$1$-absorbing prime ideal of $R$.
\end{proof}

\begin{cor}
Let $(R,\mathfrak{m})$ be a graded local ring and $\phi(I)\neq\emptyset$ for
every ideal $I$ of $R$. If $\mathfrak{m}^{3}=\{0\}$, then every proper graded
ideal of $R$ is graded $\phi$-$1$-absorbing prime.
\end{cor}

\begin{proof}
Apply Proposition \ref{Lemma 1}.
\end{proof}

\begin{thm}
\label{Theorem 1} Let $R$ be a graded ring and $I$ be a proper graded ideal of
$R$. Consider the following conditions.

\begin{enumerate}
\item $I$ is a graded $\phi$-$1$-absorbing prime ideal of $R$.

\item For each nonunits $a,b\in h(R)$ with $ab\notin I$, $(I:ab)=I\cup
(\phi(I):ab)$.

\item For each nonunits $a, b\in h(R)$ with $ab\notin I$, either $(I : ab) =
I$ or $(I : ab) =(\phi(I) : ab)$.

\item For each nonunits $a,b\in h(R)$ and proper graded ideal $L$ of $R$ such
that $abL\subseteq I$ and $abL\nsubseteq\phi(I)$, either $ab\in I$ or
$L\subseteq I$.

\item For each nonunit $a\in h(R)\ $and proper graded ideals $K,L\ $of $R$
such that $aKL\subseteq I\ $and $aKL\nsubseteq\phi(I),\ $either \ $aK\subseteq
I\ $or $L\subseteq I.$

\item For each proper graded ideals $J,K,L\ $of $R\ $such that $JKL\subseteq
I$ and $JKL\nsubseteq\phi(I),\ $either $JK\subseteq I\ $or $L\subseteq I.\ $

Then, $(6)\Rightarrow(5)\Rightarrow(4)\Rightarrow(3)\Rightarrow(2)\Rightarrow
(1).$
\end{enumerate}
\end{thm}

\begin{proof}
$(6)\Rightarrow(5):\ $Suppose that $aKL\subseteq I\ $and $aKL\nsubseteq
\phi(I)$ for some nonunit $a\in h(R)\ $and proper graded ideals $K,L\ $of
$R.\ $Then $J=Ra$ is a graded ideal since $a\in h(R)$, and also $JKL\subseteq
I$ and $JKL\nsubseteq\phi(I).\ $Then by $(6),\ $we have $aK\subseteq
JK\subseteq I$ or $L\subseteq I$ which completes the proof.

$(5)\Rightarrow(4):\ $Let $abL\subseteq I$ and $abL\nsubseteq\phi(I)$ for some
nonunits $a,b\in h(R)$ and proper graded ideal $L\ $of $R.\ $Now, put
$K=Rb.\ $Then $K\ $is a graded ideal such that $aKL\subseteq I$ and
$aKL\nsubseteq\phi(I).\ $Then by $(5),\ $we have that $ab\in aK\subseteq I$ or
$L\subseteq I\ $which is needed.

$(4)\Rightarrow(3):\ $Let $a,b\in h(R)\ $nonunits such that $ab\notin
I.\ $Then $(I:ab)\ $is a proper graded ideal of $R.$\ We have two cases.
\textbf{Case 1:} let $ab(I:ab)\subseteq\phi(I).\ $Then $(I:ab)\subseteq\left(
\phi(I):ab\right)  .\ $As the reverse inclusion always holds, we have the
equality $(I:ab)=\left(  \phi(I):ab\right)  .$ \textbf{Case 2: }let
$ab(I:ab)\nsubseteq\phi(I).\ $Since $ab(I:ab)\subseteq I,\ $by $(4),\ $we get
$(I:ab)\subseteq I.\ $As $I\subseteq(I:ab)\ $always holds, we have
$I=(I:ab).$\ Therefore, $(I:ab)=I$ or $(I:ab)=(\phi(I):ab)$.

$(3)\Rightarrow(2):\ $It is clear.

$(2)\Rightarrow(1):\ $Let $abc\in I-\phi(I)\ $for some nonunits $a,b,c\in
h(R).\ $Assume that $ab\notin I.\ $Then we have $c\in(I:ab)-\left(
\phi(I):ab\right)  .\ $By $(2),\ $we conclude that $c\in I\ $which completes
the proof.
\end{proof}

In the previous Theorem, the implication $(1)\Rightarrow(6)$ is not true in
general. See the following example.

\begin{exa}
Consider the ring $R=%
\mathbb{Z}
_{50}[X].\ $Then $R=%
{\textstyle\bigoplus\limits_{n\in\mathbb{Z} }}
R_{n}\ $is a $%
\mathbb{Z}
$-graded ring, where $R_{n}=\{\overline{0}\}\ $if $n<0,\ R_{0}=%
\mathbb{Z}
_{50}$ and also $R_{n}=%
\mathbb{Z}
_{50}X^{n}\ $if $n>0.\ $Then the set of all nonunit homogeneous elements is
$nu(h(R))=\{\overline{2k},\overline{5k},\overline{a}X^{n}:k,a\in%
\mathbb{Z}
\ $and $n\geq1\}.\ $Now, consider the graded ideal $I=(X,\overline{25})$ of
$R.\ $Set $\phi(I)=\{\overline{0}\}.\ $Now, we will show that $I\ $is a graded
$\phi$-1-absorbing prime ideal of $R.$\ To see this, choose nonunit
homogeneous elements $r,s,t\in nu(h(R))$ such that $rst\in I-\phi(I).\ $We
have two cases. \textbf{Case 1:} If at least one of the $r,s,t$ is of the form
$\overline{a}X^{n},\ $then we have $rs\in I\ $or $t\in I\ $since $X\in
I.\ $\textbf{Case 2:} Assume that $r,s,t\in\{\overline{2k},\overline{5k}:k\in%
\mathbb{Z}
\}.\ $Then we can write $r=\overline{m},s=\overline{n},t=\overline{k}\ $for
some $m,n,k\in%
\mathbb{Z}
.\ $Since $rst\in I-\phi(I),\ $we have $25|mnk\ $and $2\nmid mnk.\ $Thus,
$2\ $does not divide $m,n$ and $k.\ $Which implies that $25|mn$ and so $rs\in
I.\ $Therefore, $I\ $is a graded $\phi$-1-absorbing prime ideal of $R.\ $Now,
we will show that $I\ $does not satisfy $(2)\ $in Theorem \ref{Theorem 1}.
Now, take $a=\overline{2}\ $and $b=\overline{5}.\ $Then note that
$ab=\overline{10}\notin I.\ $Also, it is easy to see that $\overline{5}%
,X\in(I:ab).\ $Then we have $\overline{5}+X\in(I:ab).\ $On the other hand,
note that $\overline{5}+X\notin(\phi(I):ab)\cup I.\ $This shows that
$(\phi(I):ab)\cup I\subsetneq(I:ab).\ $Thus, $I\ $does not satisfy (2), and so
it does not satisfy all axioms $(2)-(6)$ in Theorem \ref{Theorem 1}.
\end{exa}

\begin{defn}
Let $R$ be a $G$-graded ring and $\phi:GI(R)\rightarrow GI(R)\cup
\{\emptyset\}$ be a function. Suppose that $g\in G$ and $I$ is graded ideal of
$R$ with $I_{g}\neq R_{g}$. Then $I$ is called a $g$-$\phi$-$1$-absorbing
prime ideal of $R$ if whenever $a,b,c\in R_{g}$ are nonunits such that $abc\in
I-\phi(I)$, then $ab\in I$ or $c\in I$.
\end{defn}

\begin{thm}
\label{Theorem 1 (1)} Let $R$ be a $G$-graded ring, $g\in G$ and $I$ be a
graded ideal of $R$ with $I_{g}\neq R_{g}$. Then the following statements are equivalent.

\begin{enumerate}
\item $I$ is a $g$-$\phi$-$1$-absorbing prime ideal of $R$.

\item For each nonunits $a,b\in R_{g}$ with $ab\notin I$, $(I:_{R_{g}%
}ab)\subseteq I\cup(\phi(I):_{R_{g}}ab)$.

\item For each nonunits $a,b\in R_{g}$ with $ab\notin I$, either $(I:_{R_{g}%
}ab)\subseteq I$ or $(I:_{R_{g}}ab)=(\phi(I):_{R_{g}}ab)$.

\item For each nonunits $a,b\in R_{g}$ and graded ideal $J$ of $R$ such that
$J_{g}\neq R_{g}$, $abJ_{g}\subseteq I$ but $abJ_{g}\nsubseteq\phi(I)$, either
$ab\in I$ or $J_{g}\subseteq I$.

\item For each nonunit $a\in h(R)$ and graded ideals $J,K$ of $R$ such that
$J_{g}\neq R_{g}$, $K_{g}\neq R_{g}$, $aJ_{g}K_{g}\subseteq I$ but
$aJ_{g}K_{g}\nsubseteq\phi(I)$, either $aJ_{g}\subseteq I$ or $K_{g}\subseteq
I$.

\item For each graded ideals $J,K,L$ of $R$ such that $J_{g}\neq R_{g}$,
$K_{g}\neq R_{g}$, $L_{g}\neq R_{g}$, $J_{g}K_{g}L_{g}\subseteq I$ but
$J_{g}K_{g}L_{g}\nsubseteq\phi(I)$, either $J_{g}K_{g}\subseteq I$ or
$L_{g}\subseteq I$.
\end{enumerate}
\end{thm}

\begin{proof}
$(1)\Rightarrow(2):$ Let $a,b\in R_{g}$ be nonunits with $ab\notin I.\ $Take
$x\in(I:_{R_{g}}ab).\ $Then we have $x\in R_{g}\ $and $abx\in I.\ $Since
$ab\notin I,\ x$ is nonunit. As $I$ is a $g$-$\phi$-$1$-absorbing prime ideal
of $R$, we conclude that $x\in I\ $or $abx\in\phi(I).\ $Which implies that
$x\in I\cup(\phi(I):_{R_{g}}ab).\ $Thus, $(I:_{R_{g}}ab)\subseteq I\cup
(\phi(I):_{R_{g}}ab).$

$(2)\Rightarrow(3):\ $Assume that $(I:_{R_{g}}ab)\subseteq I\cup
(\phi(I):_{R_{g}}ab).\ $Then by \cite{Mccoy}, $(I:_{R_{g}}ab)\subseteq I$ or
$(I:_{R_{g}}ab)\subseteq(\phi(I):_{R_{g}}ab).\ $In the first case, there is
nothing to prove. Assume that $(I:_{R_{g}}ab)\subseteq(\phi(I):_{R_{g}}%
ab).\ $Since the reverse inclusion always holds, we have the equality
$(I:_{R_{g}}ab)=(\phi(I):_{R_{g}}ab).$

$(3)\Rightarrow(4):\ $Suppose that $abJ_{g}\subseteq I$ but $abJ_{g}%
\nsubseteq\phi(I)$ for some nonunits $a,b\in R_{g}$ and graded ideal $J$ of
$R$ with $J_{g}\neq R_{g}.\ $If $ab\in I$,\ then there is nothing to prove. So
assume that $ab\notin I.\ $Since $J_{g}\subseteq(I:_{R_{g}}ab)$ and
$J_{g}\nsubseteq(\phi(I):_{R_{g}}ab),\ $by $(3),\ J_{g}\subseteq(I:_{R_{g}%
}ab)\subseteq I\ $which completes the proof.

$(4)\Rightarrow(5):$ Suppose that $aJ_{g}K_{g}\subseteq I$ and $aJ_{g}%
K_{g}\subseteq\phi(I)$. Assume that $aJ_{g}\nsubseteq I$ and $K_{g}\nsubseteq
I$. Then there exists $x\in J_{g}$ such that $ax\notin I$. Also, since
$aJ_{g}K_{g}\nsubseteq\phi(I)$, there exists $y\in J_{g}$ such that
$ayK_{g}\nsubseteq\phi(I)$. Now, assume that $axK_{g}\nsubseteq\phi(I)$. Since
$a,x$ are nonunits and $axK_{g}\subseteq I$ , we have either $ax\in I$ or
$K_{g}\subseteq I$, a contradiction. So, we get $axK_{g}\subseteq\phi(I)$.
Also, we have $a(x+y)K_{g}\subseteq I$ and $a(x+y)K_{g}\nsubseteq\phi(I)$,
which implies that $a(x+y)\in I$. Since $ayK_{g}\subseteq I$ , $ayK_{g}%
\nsubseteq\phi(I)$ and $K_{g}\nsubseteq I$ , we get $ay\in I$. Thus, we obtain
$ax\in I$ giving a contradiction.

$(5)\Rightarrow(6):$ Suppose that $J_{g}K_{g}L_{g}\subseteq I$ but $J_{g}%
K_{g}L_{g}\nsubseteq\phi(I)$ for some graded ideals $J, K$ and $L$ of $R$ with
$J_{g}\neq R_{g}$, $K_{g}\neq R_{g}$ and $L_{g}\neq R_{g}$. Assume that
$J_{g}K_{g}\nsubseteq I$ and $L_{g}\nsubseteq I$. Then there exists $b\in
J_{g}$ such that $bK_{g}\nsubseteq I$. Also, since $J_{g}K_{g}L_{g}%
\nsubseteq\phi(I)$, $aK_{g}L_{g}\nsubseteq\phi(I)$ for some $a\in J_{g}$. Then
we get $aK_{g}\subseteq I$ since $aK_{g}L_{g}\subseteq I$ and $aK_{g}%
L_{g}\nsubseteq\phi(I)$. Suppose that $bK_{g}L_{g}\nsubseteq\phi(I)$. By (5),
this gives $bK_{g}\subseteq I$ or $L_{g}\subseteq I$ , which is a
contradiction. So, $bK_{g}L_{g}\subseteq\phi(I)$. As $(a+b)K_{g}L_{g}\subseteq
I$ and $(a+b)K_{g}L_{g}\nsubseteq\phi(I)$, we have $(a + b)K_{g}\subseteq I$.
This implies $bK_{g}\subseteq I$, a contradiction.

$(6)\Rightarrow(1):$ Let $abc\in I-\phi(I)$ for some nonunits $a,b,c\in R_{g}%
$. Then $(Ra)_{g}(Rb)_{g}(Rc)_{g}\subseteq I$ and $(Ra)_{g}(Rb)_{g}%
(Rc)_{g}\nsubseteq\phi(I)$. Hence, $(Ra)_{g}(Rb)_{g}\subseteq I$ or
$(Rc)_{g}\subseteq I$ showing that $ab\in I$ or $c\in I$, as desired.
\end{proof}

\begin{defn}
Let $I$ be a $g$-$\phi$-$1$-absorbing prime ideal of $R$ and $a, b, c\in
R_{g}$ be nonunits. Then $(a, b, c)$ is called a $g$-$\phi$-$1$-triple zero of
$I$ if $abc\in\phi(I)$, $ab\notin I$ and $c\notin I$.
\end{defn}

\begin{thm}
\label{Theorem 2} Suppose that $I$ is a $g$-$\phi$-$1$-absorbing prime ideal
of $R$ and $(a, b, c)$ is a $g$-$\phi$-$1$-triple zero of $I$. Then
$abI_{g}\subseteq\phi(I)$.
\end{thm}

\begin{proof}
Now, $abc\in\phi(I)$, $ab\notin I$ and $c\notin I$. Suppose that
$abI_{g}\nsubseteq\phi(I)$. Then there exists $x\in I_{g}$ such that
$abx\notin\phi(I)$. So, $ab(c + x)\in I-\phi(I)$. If $c + x$ is unit, then
$ab\in I$, a contradiction. Now, assume that $c + x$ is nonunit and so we get
$ab\in I$ or $c\in I$, a contradiction. Thus, we have $abI_{g}\subseteq
\phi(I)$.
\end{proof}

\begin{thm}
\label{Theorem 2 (1)} Suppose that $I$ is a $g$-$\phi$-$1$-absorbing prime
ideal of $R$ and $(a, b, c)$ is a $g$-$\phi$-$1$-triple zero of $I$. If $ac,
bc\notin I$, then $acI_{g}\subseteq\phi(I)$, $bcI_{g}\subseteq\phi(I)$,
$aI_{g}^{2}\subseteq\phi(I)$, $bI_{g}^{2}\subseteq\phi(I)$ and $cI_{g}%
^{2}\subseteq\phi(I)$.
\end{thm}

\begin{proof}
Suppose that $acI_{g}\nsubseteq\phi(I)$. Then there exists $x\in I_{g}$ such
that $acx\notin\phi(I)$. This implies that $a(b + x)c\in I-\phi(I)$. If $b +
x$ is unit, then $ac\in I$ which is a contradiction. Thus $b + x$ is nonunit.
Since $I$ is a $g$-$\phi$-$1$-absorbing prime ideal, we conclude either $a(b +
x)\in I$ or $c\in I$ , which implies that $ab\in I$ or $c\in I$, a
contradiction. Thus, $acI_{g}\subseteq\phi(I)$. By using similar argument, we
have $bcI_{g}\subseteq\phi(I)$. Now, we will show that $aI_{g}^{2}%
\subseteq\phi(I)$. Suppose not. Then there exist $x, y\in I_{g}$ such that
$axy\notin\phi(I)$. It implies that $a(b + x)(c + y)\in I-\phi(I)$. If $(b+x)$
is unit, then $a(c+y)\in I$ which gives $ac\in I$, a contradiction. Similarly,
$(c+y)$ is nonunit. Then either $a(b + x)\in I$ or $c + y\in I$ implying that
$ab\in I$ or $c\in I$. Thus, we have $aI_{g}^{2}\subseteq\phi(I)$. Similarly,
we get $bI_{g}^{2}\subseteq\phi(I)$ and $cI_{g}^{2}\subseteq\phi(I)$.
\end{proof}

\begin{thm}
\label{Theorem 2 (2)} Suppose that $I$ is a $g$-$\phi$-$1$-absorbing prime
ideal of $R$ and $(a, b, c)$ is a $g$-$\phi$-$1$-triple zero of $I$. If $ac,
bc\notin I$, then $I_{g}^{3}\subseteq\phi(I)$.
\end{thm}

\begin{proof}
Suppose that $I_{g}^{3}\nsubseteq\phi(I)$. Then there exist $x, y, z\in I_{g}$
such that $xyz\notin\phi(I)$, and then $(a + x)(b + y)(c + z)\in I-\phi(I)$.
If $a + x$ is unit, then we obtain that $(b + y)(c + z) = bc + bz + cy + yz\in
I$ and so $bc\in I$, which is a contradiction. Similarly, we can show that $b
+ y$ and $c + z$ are nonunits. Then we get $(a + x)(b + y)\in I$ or $c + z\in
I$. This gives $ab\in I$ or $c\in I$, a contradiction. Hence, $I_{g}%
^{3}\subseteq\phi(I)$.
\end{proof}

\begin{thm}
\label{Theorem 3} Let $R$ be a $G$-graded ring, $g\in G$ and $x\in R_{g}$ be
nonunit. Suppose that $(0 : x)\subseteq Rx$. Then $Rx$ is a $g$-$\phi$%
-$1$-absorbing prime ideal of $R$ with $\phi\leq\phi_{2}$ if and only if $Rx$
is a $g$-$1$-absorbing prime ideal of $R$.
\end{thm}

\begin{proof}
Suppose that $Rx$ is a $g$-$\phi$-$1$-absorbing prime ideal of $R$ with
$\phi\leq\phi_{2}$. Then it is also a $g$-$\phi_{2}$-$1$-absorbing prime ideal
of $R$ by the sense of Proposition \ref{Proposition 1}. Let $abc\in Rx$ for
some nonunits $a, b, c\in R_{g}$. If $abc\notin\left(  Rx\right)  ^{2}$, then
$ab\in Rx$ or $c\in Rx$. Suppose that $abc\in\left(  Rx\right)  ^{2}$. We have
$ab(c +x)\in Rx$. If $c +x$ is unit, we are done. Hence, we can assume that $c
+x$ is nonunit. Assume that $ab(c + x)\notin\left(  Rx\right)  ^{2}$. Then we
get either $ab\in Rx$ or $c + x\in Rx$ implying $ab\in Rx$ or $c\in Rx$. Now,
assume that $ab(c + x)\in\left(  Rx\right)  ^{2}$. This gives $xab\in\left(
Rx\right)  ^{2}$ and so there exists $t\in R$ such that $xab = x^{2}t$. Thus
we have $ab-xt\in(0 : x)\subseteq Rx$. Therefore, $ab\in Rx$, as needed. The
converse is clear.
\end{proof}

\begin{rem}
Note that the condition $(0:x)\subseteq Rx$ in Theorem \ref{Theorem 3}
trivially holds for every regular element $x$.
\end{rem}

\begin{thm}
\label{Theorem 4} Let $R$ be a graded ring and $I$ be a graded ideal of $R$
with $I_{e}\neq R_{e}$. Suppose that $R_{e}$ is not local ring and
$(\phi(I):_{R_{e}}a)$ is not maximal ideal of $R_{e}$ for each $a\in I_{e}$.
Then $I$ is an $e$-$\phi$-prime ideal of $R$ if and only if $I$ is an
$e$-$\phi$-$1$-absorbing prime ideal of $R$.
\end{thm}

\begin{proof}
Suppose that $I$ is an $e$-$\phi$-$1$-absorbing prime ideal of $R$. Let
$a,b\in R_{e}$ such that $ab\in I-\phi(I)$. If $a$ or $b$ is unit, then $a\in
I$ or $b\in I$, as needed. Suppose that $a,b$ are nonunits. Since
$ab\notin\phi(I)$, $(\phi(I):_{R_{e}}ab)$ is proper. Let $\mathfrak{m}$ be a
maximal ideal of $R_{e}$ containing $(\phi(I):_{R_{e}}ab)$. Since $R_{e}$ is
not local ring, there exists another maximal ideal $\mathfrak{q}$ of $R_{e}$.
Now, choose $c\in\mathfrak{q}-\mathfrak{m}$. Then $c\notin(\phi(I):_{R_{e}%
}ab)$, and so we have $(ca)b\in I-\phi(I)$. Since $I$ is an $e$-$\phi$%
-$1$-absorbing prime ideal of $R$, we get either $ca\in I$ or $b\in I$. If
$b\in I$, then we are done. Suppose that $ca\in I$. Then as $c\notin
\mathfrak{m}$, there exists $x\in R_{e}$ such that $1+xc\in\mathfrak{m}$. Note
that $1+xc$ is nonunit. If $1+xc\notin(\phi(I):_{R_{e}}ab)$, then we have
$(1+xc)ab\in I-\phi(I)$ implying $(1+xc)a\in I$ and so $a\in I$ since $ca\in
I$. Assume that $1+xc\in(\phi(I):_{R_{e}}ab)$, that is, $ab(1+xc)\in\phi(I)$.
Choose $y\in\mathfrak{m}-(\phi(I):_{R_{e}}ab)$. Then we have $(1+xc+y)ab\in
I-\phi(I)$. On the other hand, since $1+xc+y\in\mathfrak{m}$, $1+xc+y$ is
nonunit. This implies that $(1+xc+y)a\in I$. Also, since $yab\in I-\phi(I)$,
we get $ya\in I$. Then we have $a=(1+xc+y)a-x(ca)-ya\in I$. Therefore, $I$ is
an $e$-$\phi$-prime ideal of $R$. The converse follows from Proposition
\ref{Proposition 1}.
\end{proof}

Let $R$ be a $G$-graded ring and $J$ be a graded ideal of $R$. Then $R/J$ is a
$G$-graded ring by $\left(  R/J\right)  _{g}=(R_{g}+J)/J$ for all $g\in G$.
Moreover, we have the following:

\begin{prop}
(\cite{Saber}, Lemma 3.2) Let $R$ be a graded ring, $J$ be a graded ideal of
$R$ and $I$ be an ideal of $R$ such that $J\subseteq I$. Then $I$ is a graded
ideal of $R$ if and only if $I/J$ is a graded ideal of $R/J$.
\end{prop}

For any graded ideal $J$ of $R$ define a function $\phi_{J}:GI(R/J)\rightarrow
GI(R/J)\cup\{\emptyset\}$ by $\phi_{J}(I/J)=(\phi(I)+J)/J$ where $J\subseteq
I$ and $\phi_{J}(I/J)=\emptyset$ if $\phi(I)=\emptyset$. Also, note that
$\phi_{J}(I/J)\subseteq I/J$.

\begin{thm}
\label{Theorem 5 (i)} Let $I$ be a graded $\phi$-$1$-absorbing prime ideal of
$R$. Then $I/\phi(I)$ is a graded weakly $1$-absorbing prime ideal of
$R/\phi(I)$.
\end{thm}

\begin{proof}
Let $0+\phi(I)\neq(a+\phi(I))(b+\phi(I))(c+\phi(I))=abc+\phi(I)\in I/\phi(I)$
for some nonunits $a+\phi(I), b+\phi(I), c+\phi(I)\in R/\phi(I)$. Then $a, b,
c$ are nonunits in $R$ and $abc\in I-\phi(I)$. Since $I$ is a graded $\phi
$-$1$-absorbing prime ideal of $R$, $ab\in I$ or $c\in I$, and then we get
$(a+\phi(I))(b+\phi(I))=ab+\phi(I)\in I/\phi(I)$ or $c+\phi(I)\in I/\phi(I)$.
Hence, $I/\phi(I)$ is a graded weakly $1$-absorbing prime ideal of $R/\phi(I)$.
\end{proof}

Similarly, one can prove the following:

\begin{thm}
\label{Theorem 5 (iii)} Let $I$, $J$ be two graded ideals of $R$ with
$J\subseteq I$ and $I$ be a graded $\phi$-$1$-absorbing prime ideal of $R$.
Then $I /J$ is a graded $\phi_{J}$-$1$-absorbing prime ideal of $R/J$.
\end{thm}

\begin{thm}
\label{Theorem 5 (ii)} Let $I/\phi(I)$ be a graded weakly $1$-absorbing prime
ideal of $R/\phi(I)$ and $U(R/\phi(I)) =\left\{  a + \phi(I) : a\in
U(R)\right\}  $. Then $I$ is a graded $\phi$-$1$-absorbing prime ideal of $R$.
\end{thm}

\begin{proof}
Let $a, b, c\in h(R)$ be nonunits such that $abc\in I-\phi(I)$. Then we have
$0+\phi(I)\neq(a+\phi(I))(b+\phi(I))(c+\phi(I))=abc+\phi(I)\in I/\phi(I)$.
Since $U(R/\phi(I))=\left\{  a + \phi(I):a\in U(R)\right\}  $, $a+\phi(I),
b+\phi(I), c+\phi(I)$ are nonunits in $R/\phi(I)$. Since $I/\phi(I)$ is a
graded weakly $1$-absorbing prime ideal, we have either $(a+\phi
(I))(b+\phi(I))=ab+\phi(I)\in I/\phi(I)$ or $c+\phi(I)\in I/\phi(I)$, which
implies $ab\in I$ or $c\in I$. Therefore, $I$ is a graded $\phi$-$1$-absorbing
prime ideal of $R$.
\end{proof}

Let $R$ be a $G$-graded ring and $S\subseteq h(R)$ be a multiplicative set.
Then $S^{-1}R$ is a $G$-graded ring with $(S^{-1}R)_{g}=\left\{  \frac{a}%
{s}:a\in R_{h},s\in S\cap R_{hg^{-1}}\right\}  $ for all $g\in G$. If $I$ is a
graded ideal of $R$, then $S^{-1}I$ is a graded ideal of $S^{-1}R$. Consider
the function $\phi:GI(R)\rightarrow GI(R)\cup\left\{  \emptyset\right\}  $.
Define $\phi_{S}:GI(S^{-1}R)\rightarrow GI(S^{-1}R)\cup\left\{  \emptyset
\right\}  $ by $\phi_{S}(S^{-1}I)=S^{-1}\phi(I)$ and $\phi_{S}(S^{-1}%
I)=\emptyset$ if $\phi(I)=\emptyset$. It is easy to see that $\phi_{S}%
(S^{-1}I)\subseteq S^{-1}I$.

\begin{thm}
\label{Theorem 6} Let $R$ be a graded ring and $S\subseteq h(R)$ be a
multiplicative set. If $I$ is a graded $\phi$-$1$-absorbing prime ideal of $R$
with $I\cap S=\emptyset$, then $S^{-1}I$ is a graded $\phi_{S}$-$1$-absorbing
prime ideal of $S^{-1}R$.
\end{thm}

\begin{proof}
Let $\frac{a}{s}\frac{b}{t}\frac{c}{u}\in S^{-1}I-\phi_{S}(S^{-1}I)$ for some
nonunits in $h(S^{-1}R)$. Then there exists $v\in S$ such that $vabc\in I$. If
$vabc\in\phi(I),\ $then we have $\frac{a}{s}\frac{b}{t}\frac{c}{u}=\frac
{vabc}{vstu}\in S^{-1}\phi(I)=\phi_{S}(S^{-1}I)$ which is a contradiction. So
we get $vabc\in I-\phi(I)$. Since $va,b,c$ are nonunits in $R$ and $I$ is a
graded $\phi$-$1$-absorbing prime ideal, we get $vab\in I$ or $c\in I$. This
implies $\frac{a}{s}\frac{b}{t}=\frac{vab}{vst}\in S^{-1}I$ or $\frac{c}{u}\in
S^{-1}I$. Hence, $S^{-1}I$ is a graded $\phi_{S}$-$1$-absorbing prime ideal of
$S^{-1}R$.
\end{proof}

Let $R$ and $T$ be two $G$-graded rings. Then $R\times T$ is a $G$-graded ring
by $(R\times T)_{g}=R_{g}\times T_{g}$ for all $g\in G$. Moreover, we have the following:

\begin{prop}
(\cite{Saber}, Lemma 3.12) Let $R$ and $T$ be two graded rings. Then
$L=I\times J$ is a graded ideal of $R\times T$ if and only if $I$ is a graded
ideal of $R$ and $J$ is a graded ideal of $T$.
\end{prop}

Let $R$ and $T$ be two graded rings, $\phi:GI(R)\rightarrow GI(R)\cup\left\{
\emptyset\right\}  $, $\psi:GI(T)\rightarrow GI(T)\cup\left\{  \emptyset
\right\}  $ be two functions. Suppose that $\theta:GI(R\times T)\rightarrow
GI(R\times T)\cup\left\{  \emptyset\right\}  $ is a function defined by
$\theta(I\times J)=\phi(I)\times\psi(J)$ for each graded ideals $I,J$ of $R,T$
respectively. Then $\theta$ is denoted by $\theta=\phi\times\psi$.

\begin{thm}
\label{Theorem 7 (1)} Let $R$ and $T$ be two graded rings, $\phi
:GI(R)\rightarrow GI(R)\cup\left\{  \emptyset\right\}  $, $\psi
:GI(T)\rightarrow GI(T)\cup\left\{  \emptyset\right\}  $ be two functions.
Suppose that $\theta=\phi\times\psi$. If $L=I\times J$ is a graded $\theta
$-$1$-absorbing prime ideal of $R\times T$, then $I$ is a graded $\phi$-prime
ideal of $R$ and $J$ is a graded $\psi$-prime ideal of $T$.
\end{thm}

\begin{proof}
Let $a, b\in h(R)$ such that $ab\in I-\phi(I)$. Then we have $(a, 0)(1, 0)(b,
0) = (ab, 0)\in L-\theta(L)$ for some nonunits $(a, 0), (1, 0), (b, 0)\in
h(R\times T)$. Since $L$ is a graded $\theta$-$1$-absorbing prime ideal of
$R\times T$, we get either $(a, 0)(1, 0) = (a, 0)\in L$ or $(b, 0)\in L$
implying that $a\in I$ or $b\in I$. Therefore, $I$ is a graded $\phi$-prime
ideal of $R$. Similarly, $J$ is a graded $\psi$-prime ideal of $T$.
\end{proof}

\begin{thm}
\label{Theorem 7 (2)} Let $R$ and $T$ be two graded rings, $\phi
:GI(R)\rightarrow GI(R)\cup\left\{  \emptyset\right\}  $, $\psi
:GI(T)\rightarrow GI(T)\cup\left\{  \emptyset\right\}  $ be two functions.
Suppose that $\theta=\phi\times\psi$. If $L=I\times J$ is a graded $\theta
$-$1$-absorbing prime ideal of $R\times T$ and $\theta(L_{e})\neq L_{e}$, then
$I=R$ or $J=T$.
\end{thm}

\begin{proof}
Since $\theta(L_{e})\neq L_{e}$, either $\phi(I_{e})\neq I_{e}$ or $\psi
(J_{e})\neq J_{e}$. Suppose that $\phi(I_{e})\neq I_{e}$. Then there exists
$a\in I_{e}-\phi(I_{e})$ that is $a\in I-\phi(I)$. This implies that
$(1,0)(1,0)(a,1)=(a,0)\in L-\theta(L)$. Then we have either $1\in I$ or $1\in
J$, that is $I=R$ or $J=T$. Similarly, if $\psi(J_{e})\neq J_{e}$, we have
either $I=R$ or $J=T$.
\end{proof}

\begin{thm}
\label{Theorem 7 (3)} Let $R$ and $T$ be two graded rings, $\phi
:GI(R)\rightarrow GI(R)\cup\left\{  \emptyset\right\}  $, $\psi
:GI(T)\rightarrow GI(T)\cup\left\{  \emptyset\right\}  $ be two functions.
Suppose that $\theta=\phi\times\psi$. Suppose that $L=I\times J$ is a graded
$\theta$-$1$-absorbing prime ideal of $R\times T$ and $\theta(L_{e})\neq
L_{e}.\ $If $\phi(R_{e})\neq R_{e}\ $is not a unique maximal ideal of
$R_{e}\ $and $\psi(T_{e})\neq T_{e}\ $is not a unique maximal ideal of $T_{e}%
$, then either $L=R\times J$ and $J_{e}$ is a prime ideal of $T_{e}$ or
$L=I\times T$ and $I_{e}$ is a prime ideal of $T_{e}$
\end{thm}

\begin{proof}
By Theorem \ref{Theorem 7 (2)}, we know that $I=R\ $or $J=T.\ $Without loss of
generality, we may assume that $I=R.\ $Let $xy\in J_{e}$ for some elemts
$x,y\in T_{e}$. If $x$ or $y$ is unit, we are done. So assume that $x,y$ are
nonunits in $T_{e}$. Since $\phi(R_{e})\neq R_{e}$ is not a unique maximal
ideal of $R_{e}$, there exists a nonunit element $a\in R_{e}-\phi(R_{e})$.
Then we have $(a,1)(1,x)(1,y)=(a,xy)\in L-\theta(L)$. Since $I$ is a graded
$\theta$-$1$-absorbing prime ideal of $R\times T$, we have either
$(a,1)(1,x)=(a,x)\in L$ or $(1,y)\in L$ implying $x\in J$ or $y\in J$ that is
either $x\in T_{e}\cap J=J_{e}$ or $y\in T_{e}\cap J=J_{e}$. Therefore,
$J_{e}$ is a prime ideal of $T_{e}$.
\end{proof}

\section{Graded von Neumann regular Rings}

In this section, we introduce and study the concept of graded von Neumann
regular rings. We prove that if $R$ is a graded von Neumann regular ring and
$x\in h(R)$, then $Rx$ is a graded almost $1$-absorbing prime ideal of $R$
(Theorem \ref{5}).

\begin{defn}
Let $R$ be a $G$-graded ring. Then $R$ is said to be a graded von Neumann
regular ring if for each $a\in R_{g}$ ($g\in G$), there exists $x\in
R_{g^{-1}}$ such that $a=a^{2}x$.
\end{defn}

A graded commutative ring $R$ with unity is said to be a graded field if every
nonzero homogeneous element of $R$ is unit \cite{Saber}. Clearly, every field
is a graded field, however, the converse is not true in general, see
(\cite{Saber}, Example 3.6).

\begin{lem}
\label{1} Let $R$ be a graded ring. If $r\in R_{g}$ is a unit, then $r^{-1}\in
R_{g^{-1}}$.
\end{lem}

\begin{proof}
By (\cite{Nastasescue}, Proposition 1.1.1), $r^{-1}\in h(R)$, which means that
$r^{-1}\in R_{h}$ for some $h\in G$. Now, $rr^{-1}=1\in R_{e}$ and $rr^{-1}\in
R_{g}R_{h}\subseteq R_{gh}$. So, $0\neq rr^{-1}\in R_{e}\cap R_{gh}$, which
implies that $gh=e$, that is $h=g^{-1}$. Hence, $r^{-1}\in R_{g^{-1}}$.
\end{proof}

\begin{exa}
Every graded field is a graded von Neumann regular ring. To see this, let $R$
be a graded field and $a\in R_{g}$. If $a=0$, then $x=0\in R_{g^{-1}}$
satisfies $a=a^{2}x$. If $a\neq0$, then $a$ is unit, and then by Lemma
\ref{1}, $x=a^{-1}\in R_{g^{-1}}$ with $a=a^{2}x$. Hence, $R$ is a graded von
Neumann regular ring.
\end{exa}

\begin{lem}
\label{2} If $R$ is a graded ring, then $R_{e}$ contains all homogeneous
idempotent elements of $R$.
\end{lem}

\begin{proof}
Let $x\in h(R)$ be an idempotent element. Then $x\in R_{g}$ for some $g\in G$
and $x^{2}=x$. If $x=0$, then $x\in R_{e}$ and we are done. Suppose that
$x\neq0$. Since $x^{2}=x\cdot x\in R_{g}R_{g}\subseteq R_{g^{2}}$, $0\neq x\in
R_{g}\cap R_{g^{2}}$, and then $g^{2}=g$ which implies that $g=e$, and hence
$x\in R_{e}$.
\end{proof}

\begin{prop}
\label{3} Let $R$ be a graded ring. If $R$ is a Boolean ring, then $R$ is
trivially graded.
\end{prop}

\begin{proof}
It is enough to prove that $R_{g}=\{0\}$ for all $g\neq e$. Let $g\in G-\{e\}$
and $x\in R_{g}$. Since $R$ is Boolean, $x$ is an idempotent, and then $x\in
R_{e}$ by Lemma \ref{2}. So, $x\in R_{g}\cap R_{e}$ which implies the either
$x=0$ or $g=e$. Since $g\neq e$, $x=0$, and hence $R$ is trivially graded.
\end{proof}

\begin{exa}
Every Boolean graded ring is a graded von Neumann regular ring. To see this,
let $R$ be a Boolean graded ring. Then by Proposition \ref{3}, $R$ is
trivially graded. Assume that $a\in R_{g}$. If $g\neq e$, then $a=0$ and then
$x=0\in R_{g^{-1}}$ with $a=a^{2}x$. If $g=e$, then $a$ is an idempotent, and
then $x=a\in R_{e}=R_{g^{-1}}$ with $a^{2}x=ax=a\cdot a=a^{2}=a$. Hence, $R$
is a graded von Neumann regular ring.
\end{exa}

\begin{lem}
\label{4} Let $R$ be a graded von Neumann regular ring and $x\in h(R)$. Then
$Rx=Ra$ for some idempotent element $a\in R_{e}$.
\end{lem}

\begin{proof}
Since $x\in h(R)$, $x\in R_{g}$ for some $g\in G$, and then there exists $y\in
R_{g^{-1}}$ such that $x=x^{2}y$ as $R$ is graded von Neumann regular. Choose
$a=xy$, then $a=xy\in R_{g}R_{g^{-1}}\subseteq R_{e}$, and $a^{2}=\left(
xy\right)  \cdot\left(  xy\right)  =(x^{2}y)y=xy=a$, which means that $a$ is
an idempotent. Now, $a=xy=yx\in Rx$, so $Ra\subseteq Rx$. On the other hand,
$x=x^{2}y=x(xy)=xa\in Ra$, so $Rx\subseteq Ra$. Hence, $Rx=Ra$.
\end{proof}

\begin{thm}
\label{5} Let $R$ be a graded von Neumann regular ring and $x\in h(R)$. Then
$Rx$ is a graded almost $1$-absorbing prime ideal of $R$.
\end{thm}

\begin{proof}
By [\cite{Dawwas Yildiz}, Lemma 1], $I=Rx$ is a graded ideal of $R$. By Lemma
\ref{4}, $I=Rx=Ra$ for some idempotent $a\in R_{e}$, and then $I^{2}=I$ which
implies that $I=Rx$ is a graded almost $1$-absorbing prime ideal of $R$.
\end{proof}

\begin{prop}
\label{6} Let $R$ be a graded von Neumann regular ring and $x\in h(R)$. Then
there exists an idempotent graded ideal $J$ of $R$ such that $R=Rx+J$ and
$Rx\cap J=\{0\}$.
\end{prop}

\begin{proof}
By Lemma \ref{4}, $Rx=Ra$ for some an idempotent $a\in R_{e}$. Choose
$J=R(1-a)$, then as $1-a\in R_{e}\subseteq h(R)$, $J$ is a graded ideal of $R$
by [\cite{Dawwas Yildiz}, Lemma 1]. Also, $(1-a)^{2}=1-2a+a^{2}=1-2a+a=1-a$
which means that $1-a$ is an idempotent, and so $J$ is an idempotent ideal.
Let $r\in R$. Then $r=ra+r(1-a)\in Ra+R(1-a)=Rx+J$, and hence $R=Rx+J$. Assume
that $y\in Rx\cap J=Ra\cap J$. Then $y=\alpha a$ and $y=\beta(1-a)$ for some
$\alpha,\beta\in R$. Now, $ya=\alpha a^{2}=\alpha a=y$, and $ya=\beta
(1-a)a=\beta a-\beta a^{2}=\beta a-\beta a=0$. So, $y=0$, and hence $Rx\cap
J=\{0\}$.
\end{proof}

\begin{cor}
If $R$ is a graded von Neumann regular ring, then $R$ is a direct sum of two
idempotent graded ideals of $R$.
\end{cor}

\begin{proof}
Apply Proposition \ref{6} and Lemma \ref{4}.
\end{proof}

\end{document}